\documentclass{compositio}

\usepackage{amssymb,mathrsfs}
\usepackage{amsmath}
\usepackage[normalem]{ulem}
\usepackage{amssymb,amstext,titletoc,fancyhdr,color,xcolor,calc,graphicx}

\theoremstyle{plain}
\newtheorem{theorem}{Theorem}[section]
\newtheorem{lemma}[theorem]{Lemma}

\newtheorem{corollary}[theorem]{Corollary}

\theoremstyle{definition}
\newtheorem{definition}[theorem]{Definition}
\newtheorem{example}[theorem]{Example}

\theoremstyle{remark}
\newtheorem{remark}[theorem]{Remark}

\numberwithin{equation}{section}

\def\Z{{\mathbb Z}}

\newcommand{\beqnn}{\begin{equation}}
\newcommand{\eeqnn}{\end{equation}}
\newcommand{\eb}{\begin{enumerate}}
\newcommand{\ee}{\end{enumerate}}
\newcommand{\bbm}{\begin{bmatrix}}
\newcommand{\ebm}{\end{bmatrix}}
\newcommand{\bpm}{\begin{pmatrix}}
\newcommand{\epm}{\end{pmatrix}}

\newcommand{\ra}{\rightarrow}

\newcommand{\ts}{\textstyle}
\newcommand{\bi}{\begin{itemize}}
\newcommand{\ei}{\end{itemize}}
\newcommand{\beq}{\begin{eqnarray*}}
\newcommand{\eeq}{\end{eqnarray*}}
\def\sn{\mathop{\rm sn}}  
\def\cn{\mathop{\rm cn}}  
\def\dn{\mathop{\rm dn}}

\definecolor{darkgreen}{rgb}{0,0.6,0.2}

\newcommand{\beqq}{\begin{eqnarray}}
\newcommand{\eeqq}{\end{eqnarray}}
\newcommand{\beqn}{\begin{eqnarray}}
\newcommand{\eeqn}{\end{eqnarray}}
\newcommand{\op}{\displaystyle
  \mathop{
    \vphantom{\bigoplus}
    \mathchoice
      {\vcenter{\hbox{\resizebox{\widthof{$\displaystyle\bigoplus$}}{!}{$\boxplus$}}}}
      {\vcenter{\hbox{\resizebox{\widthof{$\bigoplus$}}{!}{$\boxplus$}}}}
      {\vcenter{\hbox{\resizebox{\widthof{$\scriptstyle\oplus$}}{!}{$\boxplus$}}}}
      {\vcenter{\hbox{\resizebox{\widthof{$\scriptscriptstyle\oplus$}}{!}{$\boxplus$}}}}
  }\displaylimits
}

\begin{document}

\title[Solutions of Darboux Equations]{Solutions of Darboux Equations, its Degeneration and
        Painlev\'e VI Equations}

\author{Yik-Man Chiang}
\email{machiang@ust.hk}
\address{Department of Mathematics, Hong Kong University of Science and Technology,
Clear Water Bay, Kowloon, Hong Kong, SAR}
\author{Avery Ching}
\email{maaching@ust.hk}
\address{Department of Mathematics, Hong Kong University of Science and Technology, Clear Water Bay, Kowloon, Hong Kong, SAR}
\author{Chiu-Yin Tsang}
\email{macytsang@ust.hk}
\address{Department of Mathematics, Hong Kong University of Science and Technology, Clear Water Bay, Kowloon, Hong Kong, SAR}


\thanks{The first and third authors are partially supported by Hong Kong Research Grant Council \#16300814.}

\begin{abstract}
In this paper, we study the Darboux equations in both classical and system form, which
give the elliptic Painlev\'e VI equations by the isomonodromy deformation method. Then we
establish the full correspondence between the special Darboux equations and the special Painlev\'e VI equations.
Instead of the system form, we especially focus on the Darboux equation in a scalar form,
which is the generalization of the classical Lam\'{e} equation. We introduce a new infinite
series expansion (in terms of the compositions of hypergeometric functions and Jacobi elliptic functions)
for the solutions of the Darboux equations and regard special solutions of the
Darboux equations as those terminating series. The Darboux equations characterized in this manner have an almost (but not completely) full correspondence to
the special types of the Painlev\'e VI equations. Finally, we discuss the convergence of these infinite series expansions.
\end{abstract}

\maketitle

\vspace*{6pt}\tableofcontents

\section{Introduction}

The Lam\'e equation
\[
\dfrac{d^2 y}{du^2}+[h-k^2\alpha(\alpha+1)\ts{\sn^2(u,k)}]y=0
\]
was introduced by Lam\'e in 1837 to solve the Laplace equation in the ellipsoidal coordinates (he searched for isothermal surfaces in \cite{Lame}, and arrived at equation 17). Because of its importance in mathematical physics, the Lam\'e equations were solved by series expansions by a number of mathematicians like Ince \cite{Ince2, Ince1}, Erd\'elyi \cite{Erdelyi} and Sleeman \cite{Sleeman}.

Then in 1882, Darboux \cite{Darboux} introduced its generalization
\begin{eqnarray}\label{E:darboux}
                     &&\frac{{d}^{2}y}{{du}^{2}}+\Big(h-\frac{\xi(\xi+1)}{\sn^2(u,k)}-\frac{\eta(\eta+1)\dn^2(u,k)}{\cn^2(u,k)}-\nonumber\\
                     &&\hspace{2.5cm}\frac{\mu(\mu+1)k^2\cn^2(u,k)}{\dn^2(u,k)}-\nu(\nu+1)k^2{\ts\sn^2(u,k)}\Big)y=0,
           \end{eqnarray}
or in the following Weierstrass form
\begin{eqnarray}\label{E:darbouxweierstrass}
                     &&\frac{{d}^{2}y}{{du}^{2}}+\Big(h-\xi(\xi+1)\wp(u,\tau)-\eta(\eta+1)\wp(u+\frac{1}{2},\tau)-\nonumber\\
                     &&\hspace{2.5cm}\mu(\mu+1)\wp(u+\frac{\tau}{2},\tau)-\nu(\nu+1)\wp(u+\frac{1+\tau}{2},\tau)\Big)y=0,
           \end{eqnarray}
where $\wp(\cdot,\tau)$ is the Weierstrass elliptic function with periods $1$ and $\tau$. Some other forms of the Darboux equations are possible. For instance, they can be de-normalized to become the Sparre equations (see Matveev and Smirnov \cite{MS} for detail).

In either form, the Darboux equation is simply the pull-back of the Heun equation via a double cover. Not much was known about these linear equations. But the research was soon redirected as explained in the following paragraphs.

It is well known that two second order linear equations have their solutions related by a first order operator if they differ by a gauge transformation. Thus, a sensible strategy is to tackle a class of gauge equivalent equations rather than one particular equation. Such a gauge equivalent class is known as an isomonodromic family. In a modern language, one writes a family of Heun equations (parametrized by the cross ratio of the four singular points) as a system of first order linear PDEs
\begin{equation}\label{E:heunsystem}
dY=[A_0\dfrac{dx}{x}+A_1\dfrac{dx}{x-1}+A_2\dfrac{d(x-t)}{x-t}]Y
\end{equation}
Then the flatness condition (or integrability condition) of such a system of over-determined (two unknown functions satisfying four equations) PDEs yields the Painlev\'e VI equation in rational form. The readers may refer to the work of Jimbo and Miwa \cite{JM} for a complete treatment.

The search of special solutions of Painlev\'e VI transcendence is a prominent research subject while another more natural path is almost forgotten.
Painlev\'e (\cite{Painleve}) introduced the elliptic Painlev\'e VI equation (and studied by Manin \cite{Manin} in an algebro-geometric setup)
\begin{eqnarray}\label{E:ellipticpainleve}
\frac{{d}^{2}u(\tau)}{{d\tau}^{2}}&=&-\frac{1}{8\pi^2}\Big[a_0^2{\wp'(u(\tau);\tau)}
+{a_1^2\wp'(u(\tau)+\frac{1}{2};\tau)}\nonumber\\
&&\ +
{a_2^2\wp'(u(\tau)+\frac{\tau}{2};\tau)}+(a_3-1)^2{\wp'(u(\tau)+\frac{1+\tau}{2};\tau)}\Big]
\end{eqnarray}
by pulling back the rational Painlev\'e VI equation via a double cover. In elliptic form, the Painlev\'e VI equation shares the same properties as the rational Painlev\'e VI equation but it appears to be much more symmetric.

On the other hand, the original Darboux equation drew little attention in the century following its discovery. There is only some occasion use of it like the study of the Schwarz maps of certain Heun equations by Nehari \cite{Nehari}.

The true importance of \eqref{E:darboux} was established by Treibich and Verdier, which states that a Schr\"odinger type operator has the finite-gap property if and only if it is of the form \eqref{E:darbouxweierstrass} with integral parameters, or its degeneration (see the work of Treibich and Verdier \cite{TV, Verdier} for detail, or a concise account by Veselov \cite{Veselov}). These Darboux equations with integral parameters admit the so-called finite gap solutions \cite{Take3},\cite{smirnov1} which are integral representations of a certain kind. However, solutions of Darboux equations with general parameters are not as well understood as their integral parameters counterpart.

In this paper, we plan to
\begin{enumerate}
\item 
single out the Darboux equations which admit certain special solutions. Our method, which has its origin that at least goes back to \cite{Erdelyi1}, is to expand the solutions of the Darboux equations appropriately, such as
\[\sn(u,k)^{\xi+1}\cn(u,k)^{\eta+1}\dn(u,k)^{\mu+1}
\sum_{m=0}^{\infty}X_m\sn(u,k)^2m\sideset{_2}{_1}{\operatorname{F}}\left({\begin{matrix}
                 a+m,b+m\\
                 c+2m
                 \end{matrix}};\ \sn(u,k)^2\right).
\]
Then we seek the conditions on the parameters of the Darboux equations in order that the infinite sum above terminates. We show that these ``termination conditions" match exactly the well-known conditions given by affine Weyl group $\tilde{D}_4$ discovered by Okamoto \cite{Okamoto} where the Painlev\'e VI admits special solutions in explicit forms (see more explanation below). Moreover, we point out that these terminated sums correspond to Schlesinger transformations \cite{JM} where a general theory of linear transformations between two differential equations where their corresponding local monodromy differ by an integer was developed. These Darboux equations are summarized in Theorem \ref{terminate1}.
\item
study the convergence of the series expansions of solutions of the Darboux equations by applying the theory of three-term recursion by Poincar\'e and Perron in Section \ref{S:Convergence}.
\item
single out the elliptic Painlev\'e VI equations
which admit special solutions in Theorem \ref{riccati} (see Clarkson \cite{Clarkson} or the well-known book \cite{GLS} for detail).
\end{enumerate}
We observe that under the same conditions on the parameters, both the Darboux equations and the elliptic Painlev\'e VI equations admit special solutions. Our strategy towards an understanding of this phenomenon is to assign linear algebraic meaning to these parameters via the following steps.
\begin{enumerate}
\item
The classical Darboux equations are rewritten in a system form in Definition \ref{darbouxsystem}, so that the parameters of the classical Darboux equations become the eigenvalues of the residue matrices of this system form of the Darboux equations at the various singularities.
\item
The Darboux equations in a system form are set to undergo an isomonodromic deformation. The elliptic Painlev\'e VI equations are derived from these deformations in Section \ref{EPainleve}. The parameters of the elliptic Painlev\'e VI equation remain to be the eigenvalues of the residue matrices of the system for of the Darboux equation above at the various singularities.
\end{enumerate}
Therefore we arrive at the main observation (Corollary \ref{observation}) that under the same conditions, both the Darboux equations and the Painlev\'e VI equations have special solutions. In a certain sense, such correspondence is anticipated by Jimbo and Miwa \cite{MW}.

\section{The Darboux Connection and the Elliptic Painlev\'e VI Equation} \label{S:connection}

As we have revised in the introduction, the rational Painlev\'e VI equation comes from the flatness condition of
the system of PDEs \eqref{E:heunsystem}. Therefore, one expects the elliptic Painlev\'e VI equation \eqref{E:ellipticpainleve} to come from the flatness condition of a family of Darboux equations in system form
\[
dY=\left[A_0\dfrac{d\sigma(\wp^{-1}(x);\tau)}{\sigma(\wp^{-1}(x);\tau)}+
      A_1\dfrac{d\sigma(\wp^{-1}(x)+\frac{1}{2};\tau)}{\sigma(\wp^{-1}(x)+\frac{1}{2};\tau)}+
			A_2\dfrac{d\sigma(\wp^{-1}(x)+\frac{\tau}{2};\tau)}{\sigma(\wp^{-1}(x)+\frac{\tau}{2};\tau)}+
			 A_3\dfrac{d\sigma(\wp^{-1}(x)+\frac{1+\tau}{2};\tau)}{\sigma(\wp^{-1}(x)+\frac{1+\tau}{2};\tau)}\right]Y.
\]
We shall give a brief account of such a construction in this section.




\subsection{The Darboux Equation in a System Form} \label{DConnection}

Let $\tau\in\mathbb{H}$ be a point in the upper half plane, and
\[
k^2=\dfrac{\wp(\frac{1+\tau}{2},\tau)-\wp(\frac{1}{2},\tau)}{\wp(\frac{1+\tau}{2},\tau)-\wp(\frac{\tau}{2},\tau)}
\]
is the $\lambda$-invariant of the complex torus $\mathbb{C}/(\mathbb{Z}+\tau\mathbb{Z})$. It is well-known that the map
\[
\begin{array}{rcl}
E=\mathbb{C}/(\mathbb{Z}+\tau\mathbb{Z})&\to&\{(x:y:z)\in\mathbb{P}^2:y^2z=4x(z-x)(z-k^2x)\}\\
u&\mapsto&(\sn^2u:(\sn^2)'(u):1)
\end{array}
\]
is biholomorphic as well as a group isomorphism.

\begin{remark}
In the classical description of the Jacobi elliptic functions, the $\lambda$-invariant of a complex
torus is usually specified, whereas in the classical description of the Weierstrass elliptic functions, the period quotient of a complex torus is usually specified.
\end{remark}

The projective curve $\{(x:y:z)\in\mathbb{P}^2:y^2z=4x(z-x)(z-k^2x)\}$ admits an involution $-1\times:(x:y:z)\mapsto(x:-y:z)$.
The quotient by $\pm 1$ admits an isomorphism
\[
\begin{array}{rcl}
\{(x,y)\in\mathbb{C}\times(\mathbb{C}\backslash\{0\}):y^2=4x(1-x)(1-k^2x)\}/\{\pm 1\}&\to&\mathbb{C}\backslash\{0,1,k^{-2}\}\\
(x:y:z)&\mapsto&(x:z)
\end{array}
\]
Now there is a well-known equation called the Heun equation in system form, which is defined on the punctured Riemann sphere
$\mathbb{P}^1\backslash\{0,1,k^{-2},\infty\}$. Let $A_0$, $A_1$, $A_2$ $A_3=-A_0-A_1-A_2$ be $2\times 2$ matrices with complex
entries and consider the matrix-valued one-form
\[
\omega=A_0\dfrac{dx}{x}+A_1\dfrac{dx}{x-1}+A_2\dfrac{dx}{x-k^{-2}}.
\]
The Heun equation in system form asks for a vector-valued function $Y$ satisfying $dY=\omega Y$. The Darboux equation is indeed the pull-back
of Heun equation via the map $E\backslash\{0,\frac{1}{2},\frac{\tau}{2},\frac{1+\tau}{2}\}\stackrel{p}{\to}
\mathbb{P}^1\backslash\{0,1,k^{-2},\infty\}$. The pull-back $p^*\omega$ has a simple pole at the four points of order two.
So our target is to obtain an analogue of $dx/(x-a)$ in the torus so as to write the Darboux equation in a system form.

Recall that for each $a\in\mathbb{C}$, ``$x-a$" means the linear form
\[
\begin{array}{rcl}
\mathbb{C}^2\backslash\{(0,0)\}&\to&\mathbb{C}\\
(x,y)&\mapsto&x-ay
\end{array}
\]
which fails to be well-defined on the projective space $\mathbb{P}^1$, whereas its (unique) zero locus is well-defined (so that
this linear form is a section of a line bundle of degree one.). The Weierstrass sigma function has the same property
in the sense that $\sigma:\mathbb{C}\to\mathbb{C}$ is not periodic, but its (unique) zero locus is well-defined modulo
$\mathbb{Z}+\tau\mathbb{Z}$. Therefore, for a fixed $\tau\in\mathbb{H}$, $p^*\omega$ can be written as
\begin{eqnarray*}
p^*\omega&=&A_0\dfrac{d\sigma(z)}{\sigma(z)}+A_1\dfrac{d\sigma(z+\frac{1}{2})}{\sigma(z+\frac{1}{2})}
+A_2\dfrac{d\sigma(z+\frac{\tau}{2})}{\sigma(z+\frac{\tau}{2})}
+A_3\dfrac{d\sigma(z+\frac{1+\tau}{2})}{\sigma(z+\frac{1+\tau}{2})}\\
&=&\left(A_0\zeta(z)+A_1\zeta(z+\frac{1}{2})+A_2\zeta(z+\frac{\tau}{2})+A_3\zeta(z+\frac{1+\tau}{2})\right) dz.
\end{eqnarray*}
Recall that for each $z\in\mathbb{C}$,
\[
\zeta(z+1)=\zeta(z)+2\zeta(\frac{1}{2})\;\;\mbox{ and }\;\;\zeta(z+\tau)=\zeta(z)+2\zeta(\frac{\tau}{2}).
\]
Thus,
\begin{eqnarray*}
&&A_0\zeta(z+1)+A_1\zeta(z+\frac{1}{2}+1)+A_2\zeta(z+\frac{\tau}{2}+1)+A_3\zeta(z+\frac{1+\tau}{2}+1)\\
&=&A_0\zeta(z)+A_1\zeta(z+\frac{1}{2})+A_2\zeta(z+\frac{\tau}{2})+A_3\zeta(z+\frac{1+\tau}{2})+
   2(A_0+A_1+A_2+A_3)\zeta(\frac{1}{2})\\
&=&A_0\zeta(z)+A_1\zeta(z+\frac{1}{2})+A_2\zeta(z+\frac{\tau}{2})+A_3\zeta(z+\frac{1+\tau}{2}),
\end{eqnarray*}
as the four matrices $A_0$,..., $A_3$ are chosen so that $A_0+A_1+A_2+A_3=0$. Hence the matrix-valued form $p^*\omega$
has a period $1$. Similarly, $\tau$ is a period of $p^*\omega$ also. Consequently, $p^*\omega$ is well-defined on the
torus $E$. Therefore, we arrive at a reformulation of the Darboux equation in a system form.

\begin{definition}\label{darbouxsystem}
Let $A_0$, $A_1$, $A_2$, $A_3\in\mathfrak{s}l_2(\mathbb{C})$ such that $A_0+A_1+A_2+A_3=0$. Let
\begin{equation} \label{connection}
\Omega=A_0\dfrac{d\sigma(z)}{\sigma(z)}+A_1\dfrac{d\sigma(z+\frac{1}{2})}{\sigma(z+\frac{1}{2})}
+A_2\dfrac{d\sigma(z+\frac{\tau}{2})}{\sigma(z+\frac{\tau}{2})}
+A_3\dfrac{d\sigma(z+\frac{1+\tau}{2})}{\sigma(z+\frac{1+\tau}{2})}
\end{equation}
be a matrix-valued one-form defined on $\mathbb{C}/(\mathbb{Z}+\tau\mathbb{Z})$, except at the points of order two.
Then the Darboux equation in a system form is defined by
\[
dY=\Omega Y.
\]
\end{definition}

The correspondence between the system form and the classical form of the Darboux equation can be made explicit as follows. First of all, we need an elementary lemma.

\begin{lemma}
Let $A=(a_{ij})_{1\leq i,j\leq 2}$ be a matrix-valued holomorphic function. If $Y=(y_1,y_2)^T$ is a vector-valued function satisfying $Y'=AY$. Then,
\[
y''_1+[-a_{11}-a_{22}-\dfrac{a'_{12}}{a_{12}}]y'_1+[a_{11}a_{22}-a_{21}a_{12}-a_{12}(\dfrac{a_{11}}{a_{12}})']y_1=0.
\]
\end{lemma}
\begin{proof}
Direct calculation.
\end{proof}

In this lemma, we observe that $y_1$ satisfies a second order ODE with two kinds of singularities:
\begin{itemize}
    \item[(i)] poles of $a_{11}+a_{22}$ and that of $a_{11}a_{22}-a_{12}a_{21}$. These are called {\it essential singularities}.
    \item[(ii)] zeros or poles of $a_{12}$. These are called {\it apparent singularities}.
\end{itemize}
Note that the essential singularities of this equation are invariant under similarity transformations of $A$. The same is not true for its apparent counterpart.

Now we revise Definition \ref{darbouxsystem}. From the definition of $\Omega$ in (\ref{connection}), the first component of $Y$ satisfies a second order ODE with essential singularities $0$, $\frac{1}{2}$, $\frac{\tau}{2}$ and $\frac{1+\tau}{2}$. The (1,2) entry of $\Omega$ is an elliptic one-form with four poles, and hence it has four zeros. So it is sensible to make the following normalization:
\begin{itemize}
    \item[(i)] $A_0$, $A_1$, $A_2$ and $A_3$ have traces zero and
    \item[(ii)] the (1,2) entry of $\Omega$ has four zeros at $0$, $\frac{1}{2}$, $\frac{\tau}{2}$ and $\frac{1+\tau}{2}$
		            by applying a simultaneous similarity transformation to $A_0$, $A_1$, $A_2$ and $A_3$
\end{itemize}
so that the first component of $Y$ satisfies a second ODE with vanishing first order term. This second order ODE is nothing but the classical Darboux equation.

\begin{remark}
Instead of the Weierstrass sigma functions, the one-form $\Omega$ in (\ref{connection})
can be written in terms of the theta functions as follows
\[
A_0\dfrac{d\vartheta_1(z)}{\vartheta_1(z)}+A_1\dfrac{d\vartheta_2(z)}{\vartheta_2(z)}
+A_2\dfrac{d\vartheta_3(z)}{\vartheta_3(z)}
+A_3\dfrac{d\vartheta_4(z)}{\vartheta_4(z)}.
\]
\end{remark}

The following theorem provides the sufficient conditions for the Darboux equation having special solutions.

\begin{theorem}\label{SpSoln}
Let $A_0$, $A_1$, $A_2$, $A_3\in\mathfrak{s}l_2(\mathbb{C})$ such that $A_0+A_1+A_2+A_3=0$, and
let
$\pm a_j/2$ be the eigenvalues of $A_j$ ($j=0,1,2,3$).
If
\begin{itemize}
\item[(i)]
one of $a_0$, $a_1$, $a_2$ or $a_3$ is an integer,
then there exists $h\in\mathbb{C}$ such that the
Darboux equation
\begin{equation}\label{DarbouxEqn}
\dfrac{dY}{dz}-\left[A_0\dfrac{d\sigma(z)}{\sigma(z)}+A_1\dfrac{d\sigma(z+\frac{1}{2})}{\sigma(z+\frac{1}{2})}
+A_2\dfrac{d\sigma(z+\frac{\tau}{2})}{\sigma(z+\frac{\tau}{2})}
+A_3\dfrac{d\sigma(z+\frac{1+\tau}{2})}{\sigma(z+\frac{1+\tau}{2})}\right]Y=hY
\end{equation}
has a solution which is a sum of finitely many Gauss hypergeometric functions.
\item[(ii)]
$a_0\pm a_1 \pm a_2 \pm a_3$
is an
even
integer,
then there exists $h\in\mathbb{C}$ such that the Darboux equation (\ref{DarbouxEqn}) has a solution which is a product of powers of elliptic functions times an elliptic function with poles at the order two points.
\end{itemize}
\end{theorem}
\begin{proof}
(i) is a rephrasing of part iv) of Theorem \ref{terminate1}, while (ii) is a rephrasing of part i), ii) and iii) of Theorem \ref{terminate1}.
\end{proof}

\begin{remark}
Instead of Darboux equations, a sheaf theoretic description of two types of special solutions of Heun equations
(which are sections of sheaves constructed from rigid local systems) will be given in \cite{CCT1}.
\end{remark}


\subsection{The Correspondence to the Elliptic Painlev\'e VI Equation} \label{EPainleve}


It is well-known that the Painlev\'e VI equation comes from an isomonodromic family of Heun equations
(see \cite[Fuchs]{Fuchs}, \cite[Schlesinger]{Schlesinger}, \cite[Jimbo-Miwa]{JM}, \cite[Takemura]{Take3}).
One expects that an isomonodromic family of Darboux equations
(the monodromy of Darboux equations remains unchanged as $\tau$ varies)
yields the elliptic version of Painlev\'e VI equation
in the same manner. Our first task is to describe a family of elliptic curves.

Let $\mathbb{H}\subset\mathbb{C}$ be the upper half plane and let $g_2$, $g_3:\mathbb{H}\to\mathbb{C}$ be the
usual Eisenstein series. Let
\[
\mathbb{E}=\{(x,y,\tau)\in\mathbb{C}^2\times\mathbb{H}:y^2=4x^3-g_2(\tau)x-g_3(\tau)\}.
\]
Then $\mathbb{E}\to\mathbb{H}$ is a family of elliptic curves. Moreover,
\[
\begin{array}{l}
D_0=(\mbox{line at }\infty\times\mathbb{H})\cap\mathbb{E},\\
D_1=\{(x,y,\tau)\in\mathbb{E}:x=\wp(\frac{1}{2},\tau)\},\\
D_2=\{(x,y,\tau)\in\mathbb{E}:x=\wp(\frac{\tau}{2},\tau)\},\\
D_3=\{(x,y,\tau)\in\mathbb{E}:x=\wp(\frac{1}{2}+\frac{\tau}{2},\tau)\}
\end{array}
\]
are the four half-periods loci, so that
$\mathbb{E}\backslash (D_0\cup D_1\cup D_2\cup D_3)\to\mathbb{H}$ is the family of elliptic curves with
the four points of order two deleted. Notice that $\mathbb{E}$ admits an obvious involution
$(x,y,\tau)\mapsto(x,-y,\tau)$. The quotient of $\mathbb{E}$ by such an involution is the usual
family of $\mathbb{P}^1$ with four points deleted, as described as follows. Consider the
following hypersurfaces in $\mathbb{C}\times\mathbb{C}$,
\[
\begin{array}{l}
H_0=\{(z,t)\in\mathbb{C}\times\mathbb{C}:z=0\},\\
H_1=\{(z,t)\in\mathbb{C}\times\mathbb{C}:z=1\},\\
H_2=\{(z,t)\in\mathbb{C}\times\mathbb{C}:z=t\},
\end{array}
\]
then the projection to the second coordinate
$(\mathbb{C}\times\mathbb{C})\backslash (H_0\cup H_1\cup H_2)\to\mathbb{C}\backslash\{0,1\}$ is the
family of four-punctured Riemann sphere with one of the punctures varying. The map
\[
p:(\mathbb{E}\backslash (D_0\cup D_1\cup D_2\cup D_3))\to
(\mathbb{C}\times\mathbb{C})\backslash (H_0\cup H_1\cup H_2)
\]
is the two-to-one map described above.

Pick four matrix-valued analytic functions $A_0$, $A_1$, $A_2$,
$A_3:\mathbb{C}\backslash\{0,1\}\to\mathfrak{s}l(2)$ such that $A_0+A_1+A_2+A_3=0$, then
\[
\omega=-\left[\dfrac{A_0(t)}{x}+\dfrac{A_1(t)}{x-1}+\dfrac{A_2(t)}{x-t}\right]dx
\]
is a one-form defined on $(\mathbb{C}\times\mathbb{C})\backslash(H_0\cup H_1\cup H_2)$, and
\begin{eqnarray*}
\Omega&=&p^*\omega\\
      &=&-\left[A_0\dfrac{d\sigma(\wp^{-1}(x);\tau)}{\sigma(\wp^{-1}(x);\tau)}+
      A_1\dfrac{d\sigma(\wp^{-1}(x)+\frac{1}{2};\tau)}{\sigma(\wp^{-1}(x)+\frac{1}{2};\tau)}+
			A_2\dfrac{d\sigma(\wp^{-1}(x)+\frac{\tau}{2};\tau)}{\sigma(\wp^{-1}(x)+\frac{\tau}{2};\tau)}+
			 A_3\dfrac{d\sigma(\wp^{-1}(x)+\frac{1+\tau}{2};\tau)}{\sigma(\wp^{-1}(x)+\frac{1+\tau}{2};\tau)}\right]
\end{eqnarray*}
or equivalently,
\[
\Omega=-\sum_{j=0}^3A_j\dfrac{d\vartheta_{j+1}(\wp^{-1}(x);\tau)}{\vartheta_{j+1}(\wp^{-1}(x);\tau)}
\]
is a well-defined matrix-valued one form defined on $\mathbb{E}\backslash(D_0\cup D_1\cup D_2\cup D_3)$, with
log singularities only. The system of PDEs $dY=\omega Y$ is over-determined in general and does not have any
local solutions. However, the flatness condition $d\omega=\omega\wedge\omega$ guarantees the existence
of local solutions of $dY=\omega Y$ (see \cite{Poberezhny}). If one chooses a convenient basis so that $A_3$ is diagonal,
such a flatness condition implies that $X(t)$, the zero locus of the (1,2) entry of $\omega$, satisfies
the usual Painlev\'e VI equation \cite[\S3]{Mahoux}, \cite[\S2]{Take5}.

Similarly, the flatness condition $d\Omega=\Omega\wedge\Omega$ guarantees the existence
of local solutions of the family of Darboux equations $dY=\Omega Y$. In case $A_3$ is diagonal, the ODE
satisfied by $u(\tau)$, the zero locus of the (1,2) entry of $\Omega$, is known as the elliptic
Painlev\'e VI equation.

Note that $\Omega=p^*\omega$. The flatness condition $d\omega=\omega\wedge\omega$
implies that of $\Omega$. Thus $u(\tau)=p^*X(t)$. The elliptic Painlev\'e VI equation is the pull-back of the
rational Painlev\'e VI equation by $p$. This is not a trivial computation and the geometric detail can be found
in \cite{Manin}. The consequence is the following form of the elliptic Painlev\'e equation:

\begin{eqnarray} \label{EP6}
\frac{{d}^{2}u(\tau)}{{d\tau}^{2}}&=&-\frac{1}{8\pi^2}\Big[a_0^2{\wp'(u(\tau);\tau)}
+{a_1^2\wp'(u(\tau)+\frac{1}{2};\tau)}\nonumber\\
&&\ +
{a_2^2\wp'(u(\tau)+\frac{\tau}{2};\tau)}+(a_3-1)^2{\wp'(u(\tau)+\frac{1+\tau}{2};\tau)}\Big],
\end{eqnarray}
where
$\pm a_j/2$ are the eigenvalues of $A_j$ ($j=0,1,2,3$).
Following the idea of Manin \cite{Manin}, the elliptic Painlev\'e VI equation (\ref{EP6}) is
equivalent to the classical Painlev\'e VI equation
\begin{eqnarray} \label{P6}
\frac{d^2X}{dt^2}&=&\frac{1}{2}\bigg(\frac{1}{X}+\frac{1}{X-1}+\frac{1}{X-t}\bigg)\left(\frac{dX}{dt}\right)^2
-\bigg(\frac{1}{t}+\frac{1}{t-1}+\frac{1}{X-t}\bigg)\frac{dX}{dt}\nonumber\\&& \
+\frac{X(X-1)(X-t)}{t^2(t-1)^2}\bigg(\alpha  +\beta \frac{t}{X^2}
 \  +\gamma\frac{t-1}{(X-1)^2}+\delta\frac{t(t-1)}{(X-t)^2}\bigg),
\end{eqnarray}

where
\[(a_0^2,a_1^2,a_2^2,(a_3-1)^2)=(-2\beta,2\gamma,1-2\delta,2\alpha).\]
Moreover, the group of symmetries of (\ref{EP6}) can be generated by the following transformations
(see \cite[\S3.2]{Manin}):
\begin{itemize}
\item[(i)] $(a_i)\mapsto(-a_i)$ for $i=0,1,2$ and $(a_3-1)\mapsto(-a_3+1)$
\item[(ii)] Permutations of $(a_0,a_1,a_2,a_3-1)$;
\item[(iii)] $(a_0,a_1,a_2,a_3-1)\mapsto(a_0+n_0,a_1+n_1,a_2+n_2,a_3-1+n_3)$, where $n_0+n_1+n_2+n_3\equiv0 \mod 2$ and
             $n_i\in\mathbb{Z}.$
\end{itemize}
The following theorem gives the conditions for the Painlev\'e VI equation (\ref{P6}) having one-parameter
families of solutions expressed in terms of the hypergeometric functions (see \cite[Theorem 48.3]{GLS}):
\begin{theorem} \label{riccati}
If either
\begin{equation}\label{condition}
a_0+\epsilon_1a_1+\epsilon_2a_2+\epsilon_3a_3\in2\mathbb{Z} \mbox{ for some }
\epsilon_1,\epsilon_2,\epsilon_3\in\{1,-1\}
\end{equation}
or
\begin{equation}\label{condition1}
(a_0-n)(a_1-n)(a_2-n)(a_3-n)=0 \mbox{ for some } n\in\mathbb{Z},
\end{equation}
then the Painlev\'e VI equation (\ref{P6}) (or (\ref{EP6})) has one-parameter
families of solutions expressed in terms of the hypergeometric functions.
\end{theorem}

\begin{example}
For the case that $a_0=0,$ the Painlev\'e VI equation (\ref{EP6}) (and (\ref{P6}) resp.) has a trivial solution
$u(\tau)\equiv0$ (and $X(t)\equiv0$ resp.).
\end{example}

The following important observation can be seen from Theorem \ref{SpSoln} and Theorem \ref{riccati}:
\begin{corollary} \label{observation}
If the eigenvalues $\pm a_j/2$ of $A_j$ satisfy the condition (\ref{condition})
for some $\epsilon_1,\epsilon_2,\epsilon_3\in\{1,-1\}$
or the condition (\ref{condition1}) for some $n\in\mathbb{Z}$, then both the Darboux equation (\ref{DarbouxEqn})
and Painlev\'e VI equation (\ref{P6}) (or (\ref{EP6})) have special solutions in the sense in
Theorem \ref{SpSoln} and Theorem \ref{riccati} respectively.
\end{corollary}

%

\section{The Darboux Equation in a Scalar Form}

In the previous section, we saw that given a Darboux equation in system form, its first component satisfies the classical Darboux equation. The second component would then be determined by its first component and this second component would be a special kind of function if the first component is of the same kind. Thus, we shall focus on the first component in this section and search for special types of solutions of the classical Darboux equations.

Recall that different types of series expansions of the solution of the Lam\'{e} equation
           \begin{equation*}\label{E:lame}
                 \frac{{d}^{2}y}{{du}^{2}}+\Big(h-{\nu(\nu+1)}{\sn(u,k)^2}\Big)y=0
           \end{equation*}
had been worked on by well-known mathematicians. For instance, the power series expansions and
the Fourier-Jacobi expansions were developed by Ince in \cite{Ince2,Ince1}.	
Then Erd\'{e}lyi \cite{Erdelyi} and Sleeman \cite{Sleeman} expanded the {L}am\'e functions into series of
associated Legendre functions of the first kind and the second kind respectively.
As an example, one of the expansions was (see \cite[(12.1)]{Erdelyi})
\begin{equation}\label{E:legendre}
\sum_{m=0}^\infty C_m \Gamma(\nu-2m+1) P^{2m}_\nu\big(\cn(u,k)\big),
\end{equation}
where
\begin{eqnarray*}
P^n_\nu(z)&=&(-1)^n\frac{\Gamma(\nu+n+1)}{\Gamma(\nu-n+1)\Gamma(n+1)}\left(\frac{1-z}{1+z}\right)^\frac{n}{2}
\sideset{_2}{_1}{\operatorname{F}}\left({\begin{matrix}
                 -\nu, \nu+1\\
                 n+1
                 \end{matrix}};\ \frac{1-z}{2}\right)\\
					 &=&(-1)^n\frac{\Gamma(\nu+n+1)}{2^n\Gamma(\nu-n+1)\Gamma(n+1)}\left({1-z^2}\right)^\frac{n}{2}
             \sideset{_2}{_1}{\operatorname{F}}\left({\begin{matrix}
                 -\nu+n, \nu+n+1\\
                 n+1
                 \end{matrix}};\ \frac{1-z}{2}\right).
\end{eqnarray*}
Notice that the series expansion (\ref{E:legendre}) can also be expressed in terms of hypergeometric functions
\begin{equation*}
\sum_{m=0}^\infty C_m
\frac{\Gamma(\nu+2m+1)}{2^{2m}\Gamma(2m+1)}\sn(u,k)^m
\sideset{_2}{_1}{\operatorname{F}}\left({\begin{matrix}
                 -\nu+2m, \nu+2m+1\\
                 2m+1
                 \end{matrix}};\ \frac{1-\cn(u,k)}{2}\right).
\end{equation*}
By the quadratic tranformations of hypergeometric functions, it becomes
\begin{equation}
\sum_{m=0}^\infty C_m
\frac{\Gamma(\nu+2m+1)}{2^{2m}\Gamma(2m+1)}\sn(u,k)^m
\sideset{_2}{_1}{\operatorname{F}}\left({\begin{matrix}
                 -\nu/2+m, \nu/2+1/2+m\\
                 2m+1
                 \end{matrix}};\ \sn(u,k)^2 \right).
\end{equation}								
In Section \ref{S:series}, we will generalize the series expansion in terms of the compositions of hypergeometric functions and
Jacobi elliptic functions (see (\ref{E:2F1}) in Definition \ref{Series}) for local solutions of the
Darboux equation\footnote{Also known as {D}arboux-{T}reibich-{V}erdier equation \cite{Veselov}.}
           \begin{eqnarray}
                     &&\frac{{d}^{2}y}{{du}^{2}}+\Big(h-\frac{\xi(\xi+1)}{\sn^2(u,k)}-\frac{\eta(\eta+1)\dn^2(u,k)}{\cn^2(u,k)}-\nonumber\\
                     &&\hspace{2.5cm}\frac{\mu(\mu+1)k^2\cn^2(u,k)}{\dn^2(u,k)}-\nu(\nu+1)k^2{\ts\sn^2(u,k)}\Big)y=0
           \end{eqnarray}
on a torus of $\mathbb{C}$ modulo the lattice $\Lambda=\{m\omega_1+n\omega_2:m,n\in\mathbb{Z}\}$, which
can be specified by the Riemann $P$-scheme
	\[
		P_{\mathbb{C}\slash\Lambda}
			\begin{Bmatrix}
\ 0\ &\ K(k)\ &\ K(k)+iK'(k)\ &\ iK'(k)\ &\\
\xi+1&\eta+1&\mu+1&\nu+1&u;\, h\\
-\xi&-\eta&-\mu&-\nu&
			\end{Bmatrix},
	\]
where the entries on the top row represent the locations of the regular singularities and the entries of the next two rows under
the corresponding singularities represent the two exponents of the local solutions there. $u$, $h$ are the independent variable and
the accessory parameter respectively.

The Riemann $P$-scheme of a linear ODE is important as its local monodromies are revealed. Let $X$ be the complement of $\{0,K,iK',K+iK'\}$ in $\mathbb{C}/\Lambda$. Choose a base point $x_0\in X$ and let $V$ be the space of solutions of equation \eqref{E:darboux} in a small neighborhood of $x_0$. Then the monodromy of equation \eqref{E:darboux} is the group representation
\[
\rho:\pi_1(X,x_0)\to GL(V)
\]
defined by analytic continuation along paths. Some important partial information about such monodromy is read off from the Riemann $P$-scheme as follows.

Let $\gamma_0$, $\gamma_K$, $\gamma_{iK'}$ and $\gamma_{K+iK'}\in\pi_1(X,x_0)$ be loops based at $x_0$ which wind around $0$, $K$, $iK'$ and $K+iK'$ respectively for once. Then $\rho(\gamma_0)$, $\rho(\gamma_K)$, $\rho(\gamma_{K+iK'})$ and $\rho(\gamma_{iK'})$ have eigenvalues $e^{2\pi i(\xi+1)}$, $e^{-2\pi i\xi}$;
$e^{2\pi i(\eta+1)}$, $e^{-2\pi i\eta}$; $e^{2\pi i(\mu+1)}$, $e^{-2\pi i\mu}$;
$e^{2\pi i(\nu+1)}$, $e^{-2\pi i\nu}$ respectively. Sometimes, it is useful to projectivise $\rho$ to yield
\[
\sigma: \pi_1(X,x_0)\stackrel{\rho}{\to}GL(V)\to\mathbb{P}GL(V).
\]
(This projectivised monodromy plays an important role in the Schwarz's map.) Then the conjugacy classes
\[
\begin{array}{l}
\sigma(\gamma_0)\cong e^{2\pi i(2\xi+1)};\\
\sigma(\gamma_K)\cong e^{2\pi i(2\eta+1)};\\
\sigma(\gamma_{K+iK'})\cong e^{2\pi i(2\mu+1)};\\
\sigma(\gamma_{iK'})\cong e^{2\pi i(2\nu+1)}
\end{array}
\]
are read off from the difference of local exponents in the Riemann $P$-scheme above.

The monodromy representation of a linear ODE is important. For instance, let $L_1$ and $L_2$ be second order linear ordinary differential operators. Denote the germs of solutions of $L_1$ and $L_2$ at an ordinary point $x_0$ by $V$ and $W$ respectively. Suppose that there is a first order differential operator $G$ (called a gauge transformation) which transforms $V$ to $W$. Then such a gauge transformation induces a $\pi_1(X,x_0)$-linear map from $V$ to $W$. In particular, when $V$ and $W$ are irreducible $\pi_1(X,x_0)$ representations, the gauge transformation $G$ induces an equivalence between the $\pi_1(X,x_0)$ representations $V$ and $W$ (see \cite{Kimura1} for a nice working example).

Conversely, If $L_1$ and $L_2$ have equivalent monodromy representations, there exists a gauge transformation with rational coefficients sending germs of solutions of $L_1$ to that of $L_2$.

A special case is worth extra attention. Suppose that $L$ is a second order ordinary differential operator which has a regular singular point at $a$. Suppose further that $L$ has local exponent difference $1$ at $a$ and the local monodromy of $L$ at $a$ is diagonalizable. There is a special kind of gauge transformation called an {\it elementary Schlesinger transformation} which sends germs of solutions of $L$ to that of an operator $L'$ which does not have a singularity at $a$. The readers may refer to \cite{JM} for the detail.

\subsection{The Hypergeometric Function Series Expansion of Darboux Solutions and the Corresponding
            Special Solutions} \label{S:series}

In this section, we consider one local solution at $u=0$ with exponent $\xi+1$
and call it the {\it local Darboux solution}, denoted by $Dl(\xi,\eta,\mu,\nu;h;u,k)$.
The expansions of $Dl$ at the other regular singular points $K,\, iK^\prime,\, K+iK^\prime$ can be obtained after
applying the symmetries of the Darboux equation
(for more details, see \cite{CCT}).
If $\xi=-\frac{3}{2},-\frac{5}{2},\cdots$, then $Dl(\xi,\eta,\mu,\nu;h;u,k)$ will generically be logarithmic and
we do not discuss this degenerate case  further in this paper.

Recall the definition in \cite{CCT} that
\begin{definition}[({\cite[Definition 7.1]{CCT}})]
Suppose that $\xi\neq-\frac{3}{2},-\frac{5}{2},\cdots$. Let $Dl(\xi,\eta,\mu,\nu;h;u,k)$ be defined by the following series expansion
\begin{equation}\label{E:series}
\sn(u,k)^{\xi+1}\cn(u,k)^{\eta+1}\dn(u,k)^{\mu+1}\sum_{m=0}^\infty C_m\sn(u,k)^{2m},
\end{equation}
where the coefficients $C_m(\xi,\eta,\mu,\nu;h;k)$ satisfy the relation
{\beqn\label{3term}&&(2m+2)(2m+2\xi+3)C_{m+1}\nonumber\\
&&\hspace{.5cm} +\{h-[2m+\eta+\xi+2]^2-k^2[2m+\mu+\xi+2]^2+(k^2+1)(\xi+1)^2\}C_{m}\nonumber\\
&&\hspace{1cm} +k^2(2m+\xi+\eta+\mu+\nu+2)(2m+\xi+\eta+\mu-\nu+1)C_{m-1}=0,\nonumber\\ \eeqn}
where $m\geq0$ and the initial conditions $C_{-1}=0$, $C_0=1$.
\end{definition}

The Heun equation
\begin{equation}
		\label{heun}
			\frac{d^2y}{dt^2}+\Big(\frac{\gamma}{t}+\frac{\delta}{t-1}+\frac{\epsilon}{t-a}\Big)\frac{dy}{dt}+
      \frac{\alpha\beta t-q}{t(t-1)(t-a)}y=0,
\end{equation}
is connected to the Darboux equation by a simple change of dependent and independent variables,
and Erd\'{e}lyi \cite[Eqn(4.2)]{Erdelyi1} (1942) gave the hypergeometric function series expansion of
the local Heun solution. So it makes sense to study the corresponding expansion of the local Darboux solution
$Dl(\xi,\eta,\mu,\nu;h;u,k)$.	


{For abbreviation, we let $\op_{\pm\pm\pm\pm}$ stand for $\pm\xi\pm\eta\pm\mu\pm\nu$ such that
\begin{equation*}
\alpha=\frac{1}{2}\big(\op_{++++}+4\big),\ \beta=\frac{1}{2}\big(\op_{+++-}+3\big).
\end{equation*}
{We also let $\op_{\pm0\pm0}$ stand for $\pm\xi\pm\mu$.}

%
%
%

\begin{definition} \label{Series}
Suppose that $\xi\neq-\frac{3}{2},-\frac{5}{2},\cdots$.
Let $\tilde{Dl}(\xi,\eta,\mu,\nu;h;u,k)$ be defined by the following series expansion

           \begin{eqnarray}\label{E:2F1}
                 &&(\sn u)^{\xi+1}(\cn u)^{\eta+1}(\dn u)^{\mu+1}
                 \sum_{m=0}^\infty
                 \Gamma\Bigg(\frac{1}{2}\big(\op_{+-++}+2m+3)\Bigg)\Gamma\Bigg(\frac{1}{2}\big(\op_{+-+-}+2m+2\big)\Bigg)\nonumber\\
                 &&\times \frac{X_m(\sn u)^{2m}}{\Gamma\big(\op_{+0+0}+2m+3\big)}
                 \sideset{_2}{_1}{\operatorname{F}}\left({\begin{matrix}
                 \frac{1}{2}\big(\op_{++++}+4\big)+m, \frac{1}{2}\big(\op_{+++-}+3\big)+m\\
                 \op_{+0+0}+2m+3
                 \end{matrix}};\ \ts\sn^2 u\right),
                 \nonumber\\
           \end{eqnarray}
in which the coefficients $X_m(\xi,\eta,\mu,\nu;h;k)$ satisfy the relation
\[ L_0X_0+M_0X_1=0, \]

\begin{equation} \label{threeterm}
K_mX_{m-1}+L_mX_m+M_mX_{m+1}=0,\mbox{ for }m>0,
\end{equation}
where
\[ K_{m}=\frac{\big(\op_{++++}+2m+2\big)\big(\op_{+++-}+2m+1\big)
\big(\op_{+0+0}+m+1\big)(2\mu+2m+1)}
{2\big(\op_{+0+0}+2m+1\big)\big(\op_{+0+0}+2m\big)},\]

\begin{eqnarray*}
     L_{m}&=&\Big[\frac{\big(\op_{++++}+2m+4\big)\big(\op_{+-++}+2m+3\big)}
     {2\big(\op_{+0+0}+2m+3\big)\big(\op_{+0+0}+2m+1\big)}
     +\frac{\big(\op_{+++-}+2m+3\big)\big(\op_{+-+-}+2m+2\big)}
     {2\big(\op_{+0+0}+2m+3\big)\big(\op_{+0+0}+2m+1\big)}\\
             &&\ -\frac{2}{\big(\op_{+0+0}+2m+1\big)}\Big](2\xi+2m+1)m
             -\frac{(2\mu+3)\big(\op_{++++}+2m+4\big)
             \big(\op_{+++-}+2m+3\big)}{2\big(\op_{+0+0}+2m+3\big)}\\
             &&\ +2(2\mu+3)m+h-\xi(\xi+1)k^2+(\xi+1)(1-k^2)+2(\mu+1)(\xi+1)(1-k^2)\\
             &&\ -4k^2m\big(\op_{+0+0}+m+2\big)-\nu(\nu+1)+(\mu+1)^2(1-k^2)+2(\mu+1)(\eta+1)+\mu+\eta+2,
\end{eqnarray*}

\[ M_{m}=\frac{(m+1)(2\xi+2m+3)\big(\op_{+-++}+2m+3\big)\big(\op_{+-+-}+2m+2\big)}
{2\big(\op_{+0+0}+2m+4\big)\big(\op_{+0+0}+2m+3\big)}. \]

\end{definition}


We show that the series always converges for
\[\left|\frac{1-\cn(u,k)}{1+\cn (u,k)}\right|<\min\Big\{\big|k + ik'\big|^2,\ |k - ik'\big|^2\Big\}\]
(with the possible exception of some branch cut) in Section \ref{S:Convergence}.
Now we look at the conditions for the existence of special solutions, that is, the termination of the series. This is the case in which convergence is not an issue.


\begin{theorem}\label{terminate}
If there exists a non-negative integer $q$ such that one of the following cases
\begin{equation} \label{termination}
\op_{++++}=\xi+\eta+\mu+\nu=-2q-4 \mbox{ or }
\op_{+++-}=\xi+\eta+\mu-\nu=-2q-3
\end{equation}
or
\begin{equation} \label{termination0}
\displaystyle \mu=-\frac{2q+3}{2}
\end{equation}
holds,
then there exist $q+1$ values $h_0,\, \cdots,\, h_q$ of $h$ such that
the series expansion $\tilde{Dl}(\xi,\eta,\mu,\nu;h;u,k)$ ($j=0,1,\cdots,q$)
terminates.
\end{theorem}

\begin{proof}
It follows from $(\ref{threeterm})$ that the local Darboux solution becomes the finite series
\begin{eqnarray*}
                 &&(\sn u)^{\xi+1}(\cn u)^{\eta+1}(\dn u)^{\mu+1}
                 \sum_{m=0}^q
                 \Gamma\Bigg(\frac{1}{2}\big(\op_{+-++}+2m+3)\Bigg)\Gamma\Bigg(\frac{1}{2}\big(\op_{+-+-}+2m+2\big)\Bigg)\nonumber\\
                 &&\times \frac{X_m(\sn u)^{2m}}{\Gamma\big(\op_{+0+0}+2m+3\big)}
                 \sideset{_2}{_1}{\operatorname{F}}\left({\begin{matrix}
                 \frac{1}{2}\big(\op_{++++}+4\big)+m, \frac{1}{2}\big(\op_{++-+}+3\big)+m\\
                 \op_{+0+0}+2m+3
                 \end{matrix}};\ \ts\sn^2 u\right), \mbox{ with } X_q\neq0
                 \nonumber\\
\end{eqnarray*}
if and only if
$$K_{q+1}(\xi,\eta,\mu,\nu;k)=0=X_{q+1}(\xi,\eta,\mu,\nu;h;k).$$
Notice that $X_{q+1}(\xi,\eta,\mu,\nu;h;k)=0$ is equivalent to saying the vanishing of the finite continued-fraction
\beqn\label{fcf} \tilde{f}(\xi,\eta,\mu,\nu;h;k):=L_0/M_0-\frac{K_1/M_1}{L_1/M_1-}\frac{K_2/M_2}{L_2/M_2-}\cdots\frac{K_q/M_q}{L_q/M_q}=0.\eeqn

If $\xi,\eta,\mu,\nu$ are chosen such that $K_{q+1}(\xi,\eta,\mu,\nu;k)=0,$ that is,
(\ref{termination}) or (\ref{termination0}) holds, then
there exists $q+1$ values $h_0,\, \cdots,\, h_q$ of $h$ such that (\ref{fcf}) holds,
and hence the series terminates.
\end{proof}
\begin{remark}
If (\ref{termination}) holds, then the solution in Theorem \ref{terminate} becomes the {\it Darboux polynomial},
denoted by $Dp(\xi,\eta,\mu,\nu;h_j;u,k)$ ($j=0,\cdots,q$)
(see the details in \cite[p. 34]{CCT}.). In this case, an {\it invariant subspace} $\mathcal{V}_q$ under the Darboux operator
$\mathcal{D}$ can be constructed as follows: let the Darboux operator $\mathcal{D}$ such that the Darboux equation
(\ref{E:darboux}) can be rewritten as $(\mathcal{D}+h)y=0$ and let $\mathcal{U}_q$ be the space of all the even elliptic functions on
$\mathbb{C}\slash\Lambda$ having exactly one pole at $iK'(k)$ of order at most $2q$.
Note that $\mathcal{U}_q$ can be viewed as the space of polynomials in $\ts\sn^2 u$ with degree at most
$q$. Then it can be verified that the space
\[\mathcal{V}_q:=\{(\sn u)^{\xi+1}(\cn u)^{\eta+1}(\dn u)^{\mu+1}f:\ f\in \mathcal{U}_q\}\]
is invariant under the Darboux operator $\mathcal{D}$, that is, $\mathcal{D}g\in\mathcal{V}_q$ for all $g\in\mathcal{V}_q$. The search of the accessory parameters $h_j$ above becomes a finite dimensional eigenvalue problem.
\end{remark}

\begin{remark}
If (\ref{termination0}) holds, then we obtain a new type of the special solutions which
are finite sums of hypergeometric functions in Theorem \ref{terminate}. In this case, the Darboux equation is indeed the
pull-back of a Heun equation having an {\it apparent singularity} so that it can be reduced
to a hypergeometric equation by the gauge transformation, which is the case considered by Kimura in \cite{Kimura1}.
\end{remark}

\begin{theorem}\label{gauge}
If there exists a non-negative integer $q$ such that $\mu=-\dfrac{2q+3}{2}$, then there exists $q+1$ complex numbers $h_0$,..., $h_q$ so that $\tilde{Dl}(\xi,\eta,\mu,\nu;h_j;u,k)$ has the form of a finite sum. There also exists a first order differential operator $R$ with elliptic coefficients (which is called a gauge transformation) such that
\[
\tilde{Dl}(\xi,\eta,\mu,\nu;h_j;u,k)=R
\sideset{_2}{_1}{\operatorname{F}}\left({\begin{matrix}
                 \frac{1}{2}\big(\op_{++++}+4\big), \frac{1}{2}\big(\op_{+++-}+3\big)\\
                 \op_{+0+0}+3
                 \end{matrix}};\ \ts\sn^2 u\right).
\]
\end{theorem}
\begin{proof}
From Theorem \ref{terminate}, the Darboux equation \eqref{E:darboux} has a solution of the form $\tilde{Dl}(\xi,\eta,\mu,\nu;h;u,k)$. There also exist $h_0$,..., $h_q$ such that $\tilde{Dl}(\xi,\eta,\mu,\nu;h_j;u,k)$ is a finite sum of the form
            \begin{eqnarray}
                 &&(\sn u)^{\xi+1}(\cn u)^{\eta+1}(\dn u)^{\mu+1}
                 \sum_{m=0}^q
                 \Gamma\Bigg(\frac{1}{2}\big(\op_{+-++}+2m+3)\Bigg)\Gamma\Bigg(\frac{1}{2}\big(\op_{+-+-}+2m+2\big)\Bigg)\nonumber\\
                 &&\times \frac{X_m(\sn u)^{2m}}{\Gamma\big(\op_{+0+0}+2m+3\big)}
                 \sideset{_2}{_1}{\operatorname{F}}\left({\begin{matrix}
                 \frac{1}{2}\big(\op_{++++}+4\big)+m, \frac{1}{2}\big(\op_{+++-}+3\big)+m\\
                 \op_{+0+0}+2m+3
                 \end{matrix}};\ \ts\sn^2 u\right),
                 \nonumber\\
           \end{eqnarray}
Now by the classical contiguous relations, for each $a$, $b$, $c$ and non-negative integer $m$,
\[
\begin{array}{ll}
 &\sideset{_2}{_1}{\operatorname{F}}(a+m,b+m;c+2m;x)\\
=&\dfrac{c+2m-1}{(a+m-1)(b+m-1)}\dfrac{d}{dx}\sideset{_2}{_1}{\operatorname{F}}(a+m-1,b+m-1;c+2m-1;x)\\
=&\dfrac{c+2m-1}{(a+m-1)(b+m-1)}\dfrac{d}{dx}\\  &\left(\dfrac{c+2(m-1)}{(c-a+(m-1))(c-b+(m-1))}(1-x)\dfrac{d}{dx}-\dfrac{c-a-b}{(c-a+(m-1))(c-b+(m-1))}\right)\\
 &\sideset{_2}{_1}{\operatorname{F}}(a+m-1,b+m-1;c+2(m-1);x)\\
\end{array}
\]
Apply this relation repeatedly to the expansion above, we obtain a differential operator $L$, of order $2q$ in $\dfrac{d}{d(\sn^2 u)}$ and with elliptic coefficients such that
\[
\tilde{Dl}(\xi,\eta,\mu,\nu;h_j;u,k)=L
\sideset{_2}{_1}{\operatorname{F}}\left({\begin{matrix}
                 \frac{1}{2}\big(\op_{++++}+4\big), \frac{1}{2}\big(\op_{+++-}+3\big)\\
                 \op_{+0+0}+3
                 \end{matrix}};\ \ts\sn^2 u\right).
\]
Finally, we let $H$ be the second order differential operator which defines the hypergeometric function above. Applying the division algorithm, we obtain differential operators $Q$ and $R$ so that $R$ is of order one in $\dfrac{d}{d(\sn^2 u)}$, and
\[
L=QH+R
\]
Consequently,
\[
\tilde{Dl}(\xi,\eta,\mu,\nu;h_j;u,k)=R
\sideset{_2}{_1}{\operatorname{F}}\left({\begin{matrix}
                 \frac{1}{2}\big(\op_{++++}+4\big), \frac{1}{2}\big(\op_{+++-}+3\big)\\
                 \op_{+0+0}+3
                 \end{matrix}};\ \ts\sn^2 u\right).
\]
\end{proof}

Now we rephrase the result above for Darboux equations in system form.

\begin{corollary}
Let $A_0$, $A_1$, $A_2$, $A_3\in\mathfrak{s}l_2(\mathbb{C})$ such that $A_0+A_1+A_2+A_3=0$ and an eigenvalue of $A_2$ is a half-integer. Then there exists $h\in\mathbb{C}$ and a gauge transformation $G$ which transforms the Darboux equation
\begin{equation}\label{A}
\dfrac{dY}{dz}-[A_0\dfrac{d\vartheta_1(z)}{\vartheta_1(z)}+A_1\dfrac{d\vartheta_2(z)}{\vartheta_2(z)}
+A_2\dfrac{d\vartheta_3(z)}{\vartheta_3(z)}
+A_3\dfrac{d\vartheta_4(z)}{\vartheta_4(z)}]Y=hY
\end{equation}
to
\begin{equation}\label{B}
\dfrac{dY}{dz}-[B_0\dfrac{d\vartheta_1(z)}{\vartheta_1(z)}+B_1\dfrac{d\vartheta_2(z)}{\vartheta_2(z)}
+B_3\dfrac{d\vartheta_4(z)}{\vartheta_4(z)}]Y=h'Y
\end{equation}
for some $B_0$, $B_1$, $B_3\in\mathfrak{s}l_2(\mathbb{C})$ and $h'\in\mathbb{C}$.
\end{corollary}
\begin{proof}
Let the eigenvalues of $A_2$ be $\pm\mu$. If $\mu>-\frac{1}{2}$, then Theorem \ref{gauge} assures that there exists a gauge transformation which transforms solutions of equation \eqref{A} to that of equation \eqref{B}. If $\mu=-\frac{1}{2}$, the difference of the local exponents of equation \eqref{A} at $K+iK'$ is $1$. There exists an elementary Schlesinger transformation which transforms its solutions to that of equation \eqref{B}.
\end{proof}



%

\subsection{All other Special Solutions via Symmetries}\label{symmetries}

Using the symmetries of the Darboux Equation studied in \cite{CCT}, we can generate $2\times2\times2\times24=192$ solutions of Darboux equation (\ref{E:darboux}) in the
following form when $\xi,\eta,\mu,\nu\notin\frac{2\Z+1}{2}$:
	\[
		 \tilde{Dl}\big(\sigma_{X_i}(\xi)^{s_\xi},\sigma_{X_i}(\eta)^{s_\eta},\sigma_{X_i}(\mu)^{s_\mu},\sigma_{X_i}(\nu);h_X;\tau_{X_i}(u,k),\kappa_X(k)\big),
	\]
where ${s_\xi},{s_\eta},{s_\mu}=+$ or $-$; $X=I,\, A,\,B,\, C,\, D,\, E$; $i=0,1,2,3$,
see more details in \cite[\S8]{CCT}.

Then we obtain the following

\begin{theorem} \label{terminate1}
Assume one of the following conditions
\eb
\item $\op_{++++}\in2\Z\setminus\{-2\}$ or
\item one of $\op_{+++-}, \op_{++-+}, \op_{+-++},
\op_{-+++}$ is in $(2\Z+1)\setminus\{-1\}$ or
\item one of $\op_{++--}, \op_{+--+}, \op_{+-+-}$
 is in $2\Z\setminus\{0\}$ or
\item one of $\displaystyle \xi,\eta,\mu,\nu$ is in $\frac{2\Z+1}{2}\setminus\{-\frac{1}{2}\}$
\ee
holds. Then there exist finitely many values of $h_X$ such that at least one of the 192 local solutions in the form
$$\tilde{Dl}(\sigma_{X_i}(\xi)^{s_\xi},\sigma_{X_i}(\eta)^{s_\eta},\sigma_{X_i}
  (\mu)^{s_\mu},\sigma_{X_i}(\nu);h_X;\tau_{X_i}(u,k),\kappa_X(k))$$
terminates.
\end{theorem}
\begin{proof}
By the assumption, we choose a positive integer $q$ such that  
$$\sigma_{X_i}(\xi)^{s_\xi}+\sigma_{X_i}(\eta)^{s_\eta}+\sigma_{X_i}(\mu)^{s_\mu}+\sigma_{X_i}(\nu)=-2q-4,$$
or $$\sigma_{X_i}(\xi)^{s_\xi}+\sigma_{X_i}(\eta)^{s_\eta}+\sigma_{X_i}(\mu)^{s_\mu}-\sigma_{X_i}(\nu)=-2q-3,$$
or $$\sigma_{X_i}(\mu)^{s_\mu}=-\frac{2q+3}{2}$$
holds. Thus the proof follows from Theorem \ref{terminate}.
\end{proof}

Now we obtain the following result for Darboux equation in system form.

\begin{corollary}
Let $A_0$, $A_1$, $A_2$, $A_3\in\mathfrak{s}l_2(\mathbb{C})$ with eigenvalues $\pm\xi$, $\pm\eta$, $\pm\mu$, $\pm\nu$ respectively such that $A_0+A_1+A_2+A_3=0$. If
\eb
\item $\op_{++++}\in2\Z$ or
\item one of $\op_{+++-}, \op_{++-+}, \op_{+-++},
\op_{-+++}$ is in $(2\Z+1)$ or
\item one of $\op_{++--}, \op_{+--+}, \op_{+-+-}$
 is in $2\Z$,
\ee
then there exists $h\in\mathbb{C}$ such that the monodromy of the Darboux equation
\begin{equation}
\dfrac{dY}{dz}-[A_0\dfrac{d\vartheta_1(z)}{\vartheta_1(z)}+A_1\dfrac{d\vartheta_2(z)}{\vartheta_2(z)}
+A_2\dfrac{d\vartheta_3(z)}{\vartheta_3(z)}
+A_3\dfrac{d\vartheta_4(z)}{\vartheta_4(z)}]Y=hY
\end{equation}
is reducible. Moreover if one of $\xi$, $\eta$, $\mu$ or $\nu$ is in $\frac{2\Z+1}{2}$, there exists $h\in\mathbb{C}$ and a gauge transformation which transforms solutions of the Darboux equation above to that of a Darboux equation with one of the singularities removed.
\end{corollary}


\begin{remark}
Theorem \ref{terminate1} states that
\begin{itemize}
\item (iv) the degeneration of the projectivised local monodromy of equation \eqref{E:darboux} at a point is a necessary condition for equation \eqref{E:darboux} to have a special solution which is a finite sum of hypergeometric functions. In Theorem \ref{gauge}, this finite sum is expressed as a gauge transformation between a Darboux equation with a local exponent difference $0$ and one with local exponent difference $m\in\mathbb{N}$. Such a gauge transformation is a Schlesinger transformation which was studied in detail in \cite{JM}.
\item (i), (ii) or (iii) are the necessary conditions for the reducibility of the monodromy of equation \eqref{E:darboux}
so that a Darboux polynomial type solution is observed.
To see this, let $\rho:\pi_1(X,x_0)\to GL(2)$ be the monodromy representation of a Darboux equation which has a sub-representation spanned by a function $w$ defined around $x_0$. From this assumption, for each loop $\gamma_j$ ($j\in\{0,K,iK',K+iK'\}$), $\rho(\gamma_j)w$ is a multiple of $w$. So we infer from the Riemann $P$-scheme of the Darboux equation that
\begin{eqnarray*}
\rho(\gamma_0)w=e^{2\pi i(\pm\xi)}w;\\
\rho(\gamma_K)w=e^{2\pi i(\pm\eta)}w;\\
\rho(\gamma_{iK'})w=e^{2\pi i(\pm\nu)}w;\\
\rho(\gamma_{K+iK'})w=e^{2\pi i(\pm\mu)}w,
\end{eqnarray*}
for some choices of $\pm$.
Now we also observe from the topology of the four-punctured torus that
\[
\gamma_0\gamma_{K}\gamma_{iK'}\gamma_{K+iK'}=\delta_K\delta_{iK'}\delta^{-1}_K\delta^{-1}_{iK'}
\]
where $\delta_K$ is the straight edge joining $0$ to $K$. The same for $\delta_{iK'}$. Let both sides act on $w$, we obtain
\[
\exp(2\pi i(\op_{abcd}))w=\rho(\delta_K)\rho(\delta_{iK'})\dfrac{1}{\rho(\delta_K)}\dfrac{1}{\rho(\delta_{iK'})}w=w,
\]
for some $a$, $b$, $c$, $d\in\{\pm\}$.
Thus, we obtain $\op_{abcd}\in\mathbb{Z}$.
\end{itemize}
\end{remark}

%
%

\subsection{Convergence of the Hypergeometric Function Series Expansion} \label{S:Convergence}
In this section, we discuss the convergence when the series (\ref{E:2F1}) is non-terminating. 
\begin{theorem}\label{convergent}
Suppose that $\tilde{Dl}(\xi,\eta,\mu,\nu;h;u,k)$ is non-terminating. 
Then it converges on the domain
\[\left\{u\in\mathbb{C}:\Big|\frac{1-\cn(u,k)}{1+\cn (u,k)}\Big|<\max\Big\{\big|k + ik'\big|^{-2},\ |k - ik'\big|^{-2}\Big\}\right\}\]
if \beqn\label{ifcf} \tilde{g}(\xi,\eta,\mu,\nu;h;k):=L_0/M_0-\frac{K_1/M_1}{L_1/M_1-}\frac{K_2/M_2}{L_2/M_2-}\cdots=0\eeqn
holds. Otherwise, it converges only on the smaller domain
\[\left\{u\in\mathbb{C}:\Big|\frac{1-\cn(u,k)}{1+\cn (u,k)}\Big|<\min\Big\{\big|k + ik'\big|^{-2},\ |k - ik'\big|^{-2}\Big\}\right\}.\]
\end{theorem}

\begin{proof}
For simplicity, let the series expansion $\tilde{Dl}(\xi,\eta,\mu,\nu;h;u,k)$ be denoted by
\[(\sn u)^{\xi+1}(\cn u)^{\eta+1}(\dn u)^{\mu+1}\sum_{m=0}^\infty X_m \varphi_m(u,k).\]
As
$$
\lim_{m\ra\infty}K_m/m^2=\lim_{m\ra\infty}M_m/m^2=\frac{1}{4k^2} \mbox{ and } \lim_{m\ra\infty}L_m/m^2=\frac{1}{2k^2}-1
$$
it follows from Theorem \ref{poin}
that $\lim_{m\ra\infty}X_{m+1}/X_m$ exists, where the coefficient $X_m$ is defined in \eqref{threeterm}, and is equal to
one of the roots of the quadratic equation
$$t^2-2(2{k}^2-1)t+1=0,$$
that is,
$\lim_{m\ra\infty}X_{m+1}/X_m=(k\pm ik')^2.$
In fact, the limit depends on whether the infinite continued fraction 
$\tilde{g}(\xi,\eta,\mu,\nu;h;k)=0$
holds or not.
By Theorem \ref{perron},
$$
\lim_{m\ra\infty}|X_{m+1}/X_m|=
\begin{cases}
\min\Big\{\big|k + ik'\big|^2,\ |k - ik'\big|^2\Big\}&\mbox{ if (\ref{ifcf}) holds}\\
\max\Big\{\big|k + ik'\big|^2,\ |k - ik'\big|^2\Big\}&\mbox{ otherwise}
\end{cases}.
$$
On the other hand, it follows from Theorem \ref{asymptotic} that
\[
\frac{\varphi_{m+1}(u,k)}{\varphi_m(u,k)}\sim e^{-\zeta} \mbox{ as } m\ra\infty,
\mbox{ where }
\mbox{$\sn^2$}(u,k) = \frac{2}{1-\cosh \zeta}
\]
and hence
\[
\lim_{m\ra\infty}\frac{\varphi_{m+1}(u,k)}{\varphi_m(u,k)}=\frac{1-\cn (u,k)}{1+\cn (u,k)}.
\]
Thus by the ratio-test,
the series expansion converges for
$$\left|\frac{1-\cn (u,k)}{1+\cn (u,k)}\right|<\begin{cases}
\max\Big\{\big|k + ik'\big|^{-2},\ |k - ik'\big|^{-2}\Big\} &\mbox{ if (\ref{ifcf}) holds}\\
\min\Big\{\big|k + ik'\big|^{-2},\ |k - ik'\big|^{-2}\Big\} &\mbox{ otherwise}
\end{cases}.$$
\end{proof}

\begin{definition}
If $h=\hat{h}$ is chosen such that (\ref{ifcf}) holds, then $\tilde{Dl}(\xi,\eta,\mu,\nu;h;u,k)$ converges on the larger domain
\[\left\{u\in\mathbb{C}:\left|\frac{1-\cn (u,k)}{1+\cn (u,k)}\right|<
\max\Big\{\big|k + ik'\big|^{-2},\ |k - ik'\big|^{-2}\Big\}\right\}\]
and in this case we call the solution the {\it Darboux function}, denoted by $\tilde{Df}(\xi,\eta,\mu,\nu;\hat{h};u,k)$.
\end{definition}

\begin{remark}
When the parameters $\xi$, $\eta$, $\mu$ in the $\tilde{Df}$ become zero (or $-1$),
we recover Erd\'elyi's series expansions (see formula (12.1) in \cite{Erdelyi1}).
\end{remark}

Finally, using the symmetries of the Darboux Equation studied in \cite{CCT}, we also have

\begin{theorem} If $h$ is chosen such that $h_X(\xi,\eta,\mu,\nu;h;k)$ satisfies the infinite continued fraction
$$\tilde{g}(\sigma_{X_i}(\xi)^{s_\xi},\sigma_{X_i}(\eta)^{s_\eta},\sigma_{X_i}(\mu)^{s_\mu},\sigma_{X_i}(\nu);h_X;\kappa_X(k))=0,$$ 
then the series $\tilde{Dl}(\sigma_{X_i}(\xi)^{s_\xi},\sigma_{X_i}(\eta)^{s_\eta},\sigma_{X_i}(\mu)^{s_\mu},\sigma_{X_i}(\nu);h_X;\tau_{X_i}(u,k),\kappa_X(k))$
converges on the domain
\[\left\{u\in\mathbb{C}:\left|\frac{1-\cn (u,\kappa_X(k))}{1+\cn (u,\kappa_X(k))}\right|<
\max\Big\{\big|\kappa_X(k) + i\kappa'_X(k)\big|^{-2},\ |\kappa_X(k) - i\kappa'_X(k)\big|^{-2}\Big\}\right\}.\]
Otherwise, it converges only on the smaller domain
\[\left\{u\in\mathbb{C}:\left|\frac{1-\cn (u,\kappa_X(k))}{1+\cn (u,\kappa_X(k))}\right|<
\min\Big\{\big|\kappa_X(k) + i\kappa'_X(k)\big|^{-2},\ |\kappa_X(k) - i\kappa'_X(k)\big|^{-2}\Big\}\right\}.\]
\end{theorem}

%
%
%


\appendix
%
%
%


%
%

\section{Asymptotic Expansions of Hypergeometric Functions}

Let $\cosh\zeta=1-2x$.

\begin{theorem}[({\cite[\S 6]{Watson}})] \label{asymptotic}
\begin{eqnarray*} &&\frac{\Gamma(\alpha-\gamma+1+m)\Gamma(\gamma-\beta+m)}
                       {\Gamma(\alpha-\beta+1+2m)}x^m\sideset{_2}{_1}{\operatorname{F}}\left({\begin{matrix}
                  \alpha+m,\, \alpha-\gamma+1+m\\
                  \alpha-\beta+1+2m
                  \end{matrix}};\ x\right)\\
									&\sim&
									 2^{\alpha+\beta}x^{-\alpha}(1-e^{-\zeta})^{1/2-\gamma}(1-e^{-\zeta})^{\gamma-\alpha-\beta-1/2}
									e^{-(\alpha+m)\zeta}\sum_{s=0}^\infty c_s\Gamma(s+1/2)m^{-s-1/2}.
\end{eqnarray*}									
\end{theorem}

\section{Results on Three-Term Recursion Relations}\label{S:three term}

In this section, we review some useful results about three-term recursion relations. We refer to the readers to Gautschi \cite{Gautschi} for more details.

Given the three-term recursion
\[R_rC_{r+1}+S_rC_r+P_rC_{r-1}=0\, (r=0,1,2,\cdots),\]
where $C_{-1}=0$ and $R_r\neq0$ for all $r=0,1,2,\cdots$.
Assume that $\lim_{r\to\infty} P_r:=P,$  $\lim_{r\to\infty} S_r:=S$ and $\lim_{r\to\infty} R_r:=R$ exist.
The limit of ${C_{r+1}}/{C_r}$ ($r\to\infty$) can be determined by Poincar\'e's Theorem and Perron's Theorem.

\begin{theorem}[(Poincar\'e's Theorem (see \cite{Poincare} or {\cite[p.527]{MT}}))]\label{poin}
The limit
\[\lim_{r\to\infty} \frac{C_{r+1}}{C_r}= t_1\ \mbox{ or }\ t_2,\]
where $t_1$ and $t_2$ are the roots of the quadratic equation $Rt^2+St+P$.
\end{theorem}

\begin{theorem}[(Perron's Theorem (see {\cite[\S 57]{Perron}}))]\label{perron}
Suppose that $|t_1|<|t_2|.$ If
the infinite continued fraction
\[S_0/R_0-\frac{P_1/R_1}{S_1/R_1-}\frac{P_2/R_2}{S_2/R_2-}\cdots=0\] holds, then
\[\lim_{r\to\infty} \Bigg|\frac{C_{r+1}}{C_r}\Bigg|= t_1.\]
Otherwise, \[\lim_{r\to\infty} \Bigg|\frac{C_{r+1}}{C_r}\Bigg|= t_2.\]
\end{theorem}





%
%


\end{document}